\documentclass[12pt]{amsart}
\usepackage{hyperref}

\usepackage{amsthm, amsmath, amscd, amssymb,centernot}
\usepackage[all]{xy}
\usepackage[left=2cm,top=2.5cm,bottom=3cm,right=3cm]{geometry}
\newcommand{\li}{\mathcal L}
\newcommand{\n}{\mathcal N}
\newcommand{\nc}{\n_{\complex}}
\newcommand{\complex}{\mathbb{C}}
\theoremstyle{plain}
\numberwithin{equation}{section}
\newtheorem{theorem}{Theorem}[section]
\newtheorem{corollary}[theorem]{Corollary}

\newtheorem{lemma}[theorem]{Lemma}

\newtheorem{remark}[theorem]{Remark}

\begin{document}

\title {Syzygies of some GIT quotients}

\author[Krishna Hanumanthu]{Krishna Hanumanthu}
\address{Chennai Mathematical Institute, H1 SIPCOT IT Park, Siruseri, Kelambakkam 603103, India}
\email{krishna@cmi.ac.in}

\author[S. Senthamarai Kannan]{S. Senthamarai Kannan}
\address{Chennai Mathematical Institute, H1 SIPCOT IT Park, Siruseri, Kelambakkam 603103, India}
\email{kannan@cmi.ac.in}
\date{July 24, 2015}

\maketitle

\begin{abstract}
Let $X$ be flat scheme over $\mathbb{Z}$ such that its base change,
$X_p$, to $\overline{\mathbb{F}}_p$ is Frobenius split for all primes
$p$. Let $G$ be a reductive group scheme over $\mathbb{Z}$ acting on
$X$. In this paper, we prove a result on 
the $N_p$ property 
for line bundles on GIT
quotients of $X_{\complex}$ for the action of $G_{\complex}$. We
apply our result to the special cases of (1) an action of a finite group on the
projective space and (2) the action of a maximal torus on the flag
variety of type $A_n$. 
\end{abstract}

%\keywords
\section{Introduction}
Syzygies of algebraic varieties have been studied classically since
the time of 
Italian geometers. For instance, the question of projective normality
and normal presentation of embeddings of projective varieties in a
projective space was studied in depth. 
The subject has been revived and there is much
renewed interest since Green \cite{Green1,Green2} developed a homological
framework which encompasses the classical questions. It was noted
that projective normality and normal presentation were really
properties of a graded free resolution and $N_p$ property was defined
as a generalisation of this
property. 

We briefly review the notion of $N_p$ property. 

Let $k$ be an algebraically closed field of characteristic 0. All our
varieties are 
projective, smooth and defined over $k$. 

Let $\li$ be a very ample line bundle on a projective variety
$X$. Then $\li$ determines an embedding of $X$ into the projective
space $\mathbb{P}\big{(}H^0(X,\li)\big{)}$.  We denote by $S$ the homogeneous coordinate ring of this projective space. Then the {\it section ring} 
$R(\li)$ of $\li$ is defined as $\bigoplus_{n=0}^{\infty} H^0(X, \li^{\otimes n})$ and it is a finitely generated graded $S$-module.  One looks at the minimal graded free resolution of $R(\li)$ over $S$:
$$...\rightarrow E_i \rightarrow ... \rightarrow E_2  \rightarrow E_1 \rightarrow E_0 \rightarrow R(\li) \rightarrow 0$$
where $E_i  = \bigoplus S(-a_{i,j})$ for all $i \geq 0$ and $a_{i,j}$ are some nonnegative integers.

We say that $\li$ has $N_0$ property if $E_0 = S$. This simply means
that the embedding determined by $\li$ is projectively normal (or $\li$ is {\it normally generated}).

$\li$ is said to have $N_1$ property if $E_0 = S$ and $a_{1,j} = 2$ for all $j$. In this case, we also say that $\li$ is {\it normally presented.}  Geometrically, this means that the embedding is cut out by quadrics. 

For $p \geq 2$, we say that $\li$ has $N_p$ property if $E_0 = S$ and $a_{i,j} = i+1$ for all $i = 1, \ldots ,p$. 

Given a very ample line bundle $\li$, it is an interesting question to ask
whether it has $N_p$ property for a given $p$. 

There is extensive literature on this question. The following is a
sample of results. For line bundles on curves see \cite{GL}, on surfaces see
\cite{GP1,GP2}. For abelian varieties, see \cite{Par}. In \cite{E-L},
a general result is proved for very ample line bundles on projective varieties. 

In this paper, we are interested in the $N_p$ property for line bundles on
GIT quotients. More specifically, we consider varieties defined over
$\mathbb{Z}$
and
consider the descent of an ample line bundle to a GIT
quotient. We obtain a general result on $N_p$ property of this
descent (Corollary \ref{main}) by  using a cohomological criterion for
$N_p$ property. We 
  prove the required vanishing results using Frobenius splitting
  methods (Theorem \ref{van}). 

In \cite{KPP}, the authors consider the quotients of a projective
space $X$ for the linear action of finite solvable
groups and for 
finite groups acting by pseudo reflections. They prove that the
descent of 
$\mathcal{O}_X(1)^{\otimes |G|}$ is projectively normal. 
In
\cite{Kang}, these results were obtained for every finite group but with
a larger 
power of the descent of $\mathcal{O}_X(1) ^{\otimes |G|}$. In this
paper, we consider any finite group acting linearly on $X$ and prove a 
general result on $N_{p}$ property for the descent of $\mathcal{O}_X(1) ^{\otimes |G|}$. 
For representations $\rho: G \rightarrow GL(V)$ such that 
$\rho(G) \subset SL(V)$, it follows from 
\cite[Theorem 1]{Wat} that  $\mathbb{C}[V]^{G}$ is Gorenstein (cf
\cite[Corollary 8.2]{Stan} also). Thus $G \backslash \mathbb{P}(V)$
is also Gorenstein. In this case, with the further assumption that
the  order of
$G$ divides the dimension of $V$, we prove that the canonical line bundle on $\mathbb{P}(V)$ descends to the quotient 
$G \backslash \mathbb{P}(V)$ and it is the canonical line bundle.

A question of Fulton concerns the 
$N_p$ property of line bundles on flag varieties
(cf. \cite[Problem 4.5]{E-L}). The special case of  flag
varieties of type $A_n$ is considered in \cite{Man} and a general result is obtained.
In line with this, we consider the GIT quotient of a 
flag variety of type $A_{n}$ for the action of a maximal torus and we
obtain a result on $N_p$ property as an
application of our main result.

The organisation of the paper is as follows:

Section 2 consists of preliminaries. Cohomology of line bundles on the
quotient variety is studied in Section 3. In section 4, we prove $N_{p}$ property 
for GIT quotients of varieties which are defined over $\mathbb{Z}$. 
We apply these results to the special case of finite group quotients
in Section 5 and to 
the special case of GIT  
quotient of a flag variety of type $A_{n}$ for the action of a maximal
torus in section 6.      
\section{Preliminaries}

%Let $X$ be a smooth projective variety over $\mathbb{C}$ which is defined 
%over $\mathbb{Z}$ such that the inverse $K_{X}^{-1}$ of the canonical line 
%bundle $K_{X}$ is ample. Let $k$ be an algebraically closed field of positive 
%characteristic. Assume that the $k$- valued points $X_{k}$ is Frobenius split. 

%Let $G$ be a reductive algebraic group over $\mathbb{C}$ which is defined over 
%$\mathbb{Z}$ acting on $X$ such that $K_{X}^{-1}$ is $G$- linearised and 
%$G_{k}$- is linearly reductive over  $k$.  We further assume that $K_{X}$
%descends to the GIT quotient $G\backslash\backslash X_{G}^{ss}(K_{X})$.
%Let $\mathcal{L}=K_{X}^{-1}$.

Given a vector bundle $F$ on a projective variety $X$ that is
generated by its global sections, we have the canonical 
surjective map:\\
\begin{eqnarray}
H^0(F) \otimes {\mathcal O}_{X} \rightarrow F.
\end{eqnarray}

Let $M_F$ be the kernel of this map.  We have then the natural exact sequence:
\begin{eqnarray}\label{cansurj}
0 \rightarrow M_F \rightarrow H^0(F) \otimes {\mathcal O}_{X} \rightarrow F \rightarrow 0.
\end{eqnarray}

Our goal in this paper is to study $N_p$ property of line bundles on
GIT quotients of projective varieties with some special property. 

\begin{theorem}\label{np}\textup{\cite[Lemma 1.6]{E-L}}
Let $L$ be a very ample line bundle on a projective variety
$X$. Assume that $H^1(L^{\otimes k}) = 0$ for all $k \ge 1$. 
Then $L$  satisfies $N_p$ property 
if and only if $H^1(\wedge^{m}M_L \otimes L^{\otimes n}) = 0$ for all $1 \le m \le p+1$ and $n \ge 1$. 
\end{theorem}

\begin{remark} 
In characteristic zero, it suffices to prove $H^1(M^{\otimes m}_L
\otimes L^{\otimes n}) = 0$ to obtain $N_p$ property as the wedge
product $\wedge^{m}M_L$ is a direct summand of the tensor product
$M^{\otimes m}_L$.
\end{remark}

\begin{remark}\label{rmk}
In \cite{E-L}, this theorem was proved only assuming that $L$ is ample
and base point free. We will apply this result with only these
hypotheses (cf. \cite[\S 1.3, Page 509]{L}).
\end{remark}

\begin{lemma}\label{horace}
Let $E$ and $L_1, L_2,...,L_r$ be coherent sheaves on a variety $X$. Consider the multiplication maps

$\psi: H^0(E) \otimes H^0(L_1 \otimes ...\otimes L_r) \rightarrow H^0(E \otimes L_1 \otimes ... \otimes L_r)$,

$\alpha_1: H^0(E) \otimes H^0(L_1) \rightarrow H^0(E \otimes L_1)$,

$\alpha_2: H^0(E \otimes L_1) \otimes H^0(L_2) \rightarrow H^0(E \otimes L_1 \otimes L_2)$,

...,

$\alpha_r: H^0(E\otimes L_1 \otimes ... \otimes L_{r-1}) \otimes H^0(L_r) \rightarrow H^0(E \otimes L_1 \otimes ... \otimes L_{r})$.

If $\alpha_1$,...,$\alpha_{r}$ are surjective, then so is $\psi$.
\end{lemma}
\begin{proof}
We have the following commutative diagram where $id$ denotes the identity morphism:
\begin{displaymath}
\xymatrix @R=2pc @C=3pc{
H^0(E) \otimes H^0(L_1) \otimes ... \otimes H^0(L_r) \ar[r]^{\alpha_1 \otimes id} \ar[d]^{\phi} & H^0(E \otimes L_1) \otimes H^0(L_2) \otimes ... \otimes H^0(L_r) \ar[d]^{\alpha_2 \otimes id} \\
H^0(E) \otimes H^0(L_1 \otimes ... \otimes L_r)    \ar[dd]^{\psi} & H^0(E \otimes L_1 \otimes L_2) \otimes H^0(L_3) \otimes ... \otimes H^0(L_r)  \ar[d]^{\alpha_3 \otimes id} \\
   &...\ar[d]^{\alpha_{r-1} \otimes id}\\
H^0(E \otimes L_1 \otimes ... \otimes L_r)   &\ar[l]_{\alpha_r} H^0(E \otimes L_1 \otimes ... \otimes L_{r-1}) \otimes H^0(L_r)}
\end{displaymath}

Since $\alpha_1, \alpha_2,...,\alpha_r$ are surjective and this diagram is commutative, a simple diagram chase shows that  $\psi$ is surjective.
\end{proof}

The following result, known as Castelnuovo - Mumford lemma, will be used 
often in this paper. 
%We remark that
%though Mumford stated this fact under the hypothesis 
%that $E$ is ample and base point free, the proof works with only the base point free assumption. 

\begin{lemma}\label{key} \textup{\cite[Theorem 2]{Mum}}
Let $E$ be an ample and base-point free line bundle on a projective variety $X$ and let $F$ be a coherent sheaf on $X$. If $H^i(F \otimes E^{-i}) = 0$ for $i \geq 1$, then the multiplication map 
$$H^0(F \otimes E^{\otimes i}) \otimes H^0(E) \rightarrow H^0(F \otimes E^{\otimes i+1})$$
is surjective for all $i \geq 0$. 
\end{lemma}		

%We refer to this result as CM lemma in this paper. 

\section{Cohomology of the quotient variety}

Let $X$ be a flat scheme over $\mathbb{Z}$.
Let $p$ be a prime number and 
let $\overline{\mathbb{F}}_p$ denote the algebraic closure of the
finite field $\mathbb{F}_p$.  Let $X_p$ denote the
$\overline{\mathbb{F}}_p$-valued points of $X$. Let $X_{\mathbb{C}}$
denote the $\mathbb{C}$-valued points of $X$. We assume that 
$X_{\mathbb{C}}$ is a projective variety over $\complex$ and that
$X_p$ are projective varieties over 
$\overline{\mathbb{F}}_p$ for all primes. 

We assume that there is a sheaf $\mathcal{N}$ on $X$ such that the
base change of $\mathcal{N}$ to $X_{\complex}$,
 $\mathcal{N}_{\mathbb{C}}$, 
(respectively, $\mathcal{N}_p$ on $X_p$, for all primes) is an ample line bundle.

% the canonical line bundle $K_{X_{\mathbb{C}}}$ of $X_{\mathbb{C}}$ (respectively,  the canonical line bundle $K_{X_{p}}$ of $X_p$). We also assume that $K_{X_{\mathbb{C}}}^{-1}$ (respectively $K_{X_p}^{-1}$) is ample. 

Finally assume that $X_p$ is Frobenius split for all 
primes.

Let $G$ be a reductive (not necessarily connected) 
algebraic group scheme over $\mathbb{Z}$ 
acting on $X$ such that the action map 
$G_{\mathbb{C}} \times X_{\mathbb{C}} \longrightarrow X_{\mathbb{C}}$ is a morphism. 
Assume that every line bundle on $X_\complex$ is
$G_{\complex}$-linearised and 
that 
%$G_p$ is linearly reductive over  $\overline{\mathbb{F}}_p$  
$(X_{\mathbb{C}})_{G_{\mathbb{C}}}^{ss}(\mathcal{N}_{\complex})$ is nonempty.
We assume that the above hypotheses also hold for base change over 
$\overline{\mathbb{F}}_p$ for all but finitely many primes. 

%Further, we
%assume that $G_p$ is linearly reductive for all but finitely many
%primes.  

Let $Y_{\mathbb{C}}$ denote the  GIT quotient 
$G_{\mathbb{C}}\backslash\backslash
(X_{\mathbb{C}})_{G_{\mathbb{C}}}^{ss}(\n_{\complex})$. 
Similarly let $Y_p$ denote the GIT quotient of $X_p$ with respect to the
$G_p$-linearised line bundle $\n_p$.
We further assume that $\n_{\complex}$ (respectively, $\n_p$) 
descends to $Y_{\mathbb{C}}$ (respectively, $Y_p$ for all primes).
Let $\mathcal{L}_{\mathbb{C}}$ (respectively, $\mathcal{L}_p$) 
denote the descent of $\n_{\complex}$ to $Y_{\mathbb{C}}$ (respectively, 
$\n_p$ to $Y_p$).

For the preliminaries and notion of Geometric Invariant Theory, we refer
to \cite{GIT} and \cite{Newst}. For the notion of Frobenius
splitting, see \cite{MR}.

\begin{lemma}\label{ample}
 $\li_{\complex}$ and $\li_p$ are ample
line bundles on $Y_{\complex}$ and $Y_p$ respectively.
\end{lemma}
\begin{proof}
%This follows from 
%\cite[Theorem 1.10, Page 38]{GIT} and \cite[Proposition 4.2, Page 83]{Kraft}.

Let
$\phi:(X_{\mathbb{C}})_{G_{\mathbb{C}}}^{ss}(\n_{\complex})
\rightarrow Y_{\complex}$ be the natural categorical quotient map and let 
$\phi^{\star}: Pic(Y_{\complex}) \rightarrow Pic((X_{\mathbb{C}})_{G_{\mathbb{C}}}^{ss}(\n_{\complex}))$ be
the pullback map. 

Since $\nc$ is a $G_{\complex}$-linearised line bundle
on $X_{\complex}$, by \cite[Theorem 1.10, Page 38]{GIT}, there is an
ample line bundle $\mathcal{M}$ on $Y_{\complex}$ such that the
$\phi^{\star}(\mathcal{M}) = \nc^{\otimes n}$ for some $n > 0$.  

Since $\phi^{\star}(\li_{\complex}) =\nc$, $\mathcal{M}
\otimes \li_{\complex}^{-n}$ is in the kernel of $\phi^{\star}$. Since every line
bundle on $X_{\complex}$ is $G_{\complex}$-linearised, 
$ Pic\Huge{(}(X_{\mathbb{C}})_{G_{\mathbb{C}}}^{ss}(\nc)\huge{)}
=
Pic_{G_{\complex}}((X_{\mathbb{C}})_{G_{\mathbb{C}}}^{ss}(\nc))$.
By \cite[Proposition 4.2, Page
83]{Kraft}, $\phi^{\star}$ is injective. Hence $\mathcal{M} =
\li_{\complex}^{\otimes n}$ and $\li_{\complex}$ is ample. 

Proof is similar for $\li_p$.
\end{proof}
\begin{theorem}\label{van} With the notation as above, the following statements hold.
\begin{enumerate}
\item $H^{i}(Y_{\mathbb{C}}, \mathcal{L}_{\mathbb{C}}^{e}) = 0$ for every 
$e > 0$ and $i \ge 1$,
\item Assume that $G_p$ is linearly reductive for all but finitely many
primes.  Then 
$H^{i}(Y_{\mathbb{C}}, \mathcal{L}_{\mathbb{C}}^{e})=0$ for every $e <
0$ and $i < d$, where $d$ denotes the dimension of $Y$.
\end{enumerate}
\end{theorem} 
\begin{proof}
%For simplicity of notation in this proof we will use $X, Y$ and $\mathcal{L}$ in place of 
%$X_{\mathbb{C}}, Y_{\mathbb{C}}$ and $\mathcal{L}_{\mathbb{C}}$.  

%By abuse of notation, we will also denote the line  bundle $K_X^{-1}$ on $X$ by $\li$ in this proof (similarly $\li_p$ and $\li_{\mathbb{C}}$ will stand for appropriate base changes).

Since $\n_p$ is ample, by Serre's vanishing theorem, there is a positive integer $r$ such that 
$H^{i}(X_{p}, \n_p^{\otimes p^{r}})=0$ for $i\geq 1$. 

Now we will use the Frobenius splitting property of $X_{p}$ to prove
(1) (cf. \cite{MR}). Let $F$ denote the Frobenius morphism corresponding to prime $p$.

Tensoring the map $\mathcal{O}_{X_{p}} \longrightarrow F_{*}\mathcal{O}_{X_{p}}$
by $\n_p$ and noting that $\n_p\otimes
F_{*}\mathcal{O}_{X_{p}} \cong
F_{*}F^{*}\n_p=F_{*}\n_p^{\otimes p}$ (projection formula) we see that the map
$H^{i}(X_{p}, \n_p) \longrightarrow 
H^{i}(X_{p}, F_{*}\n_p^{\otimes p}) = 
H^{i}(X_{p}, \n_p^{\otimes p})$
is injective.

Iterating this process, we conclude that the map 
$H^{i}(X_{p}, \n_p)\longrightarrow H^{i}(X_{p}, \n_p^{\otimes p^{r}})$ 
is injective. 

Thus $H^{i}(X_{p}, \n_p)=0$ for $i\geq 1$.

Since $X$ is flat over $\mathbb{Z}$, 
using semicontinuity theorem \cite[Theorem III.12.8]{Har}, 
%{BK, Proposition 1.6.2], 
we conclude that 
$H^{i}(X_{\mathbb{C}}, \nc)=0$ for $i\geq 1$
(cf.  \cite[Proposition 1.6.2]{BK}).

Proof of 
$H^{i}(X_{\mathbb{C}}, \nc^{\otimes e})=0$ for $e\geq 2$ and for every $i\geq 1$ is similar.

Since  
$(X_{\mathbb{C}})_{G_{\mathbb{C}}}^{ss}(\nc)$ is
nonempty, using
 \cite[Theorem 3.2.a]{Tel}, we have 
$H^{i}(Y_{\mathbb{C}},
\mathcal{L}_{\mathbb{C}}^{\otimes e})=H^{i}(X_{\mathbb{C}}, \nc^{\otimes e})^{G_{\complex}}$ 
for every $i > 0$ and $e \geq 0$. 

Hence, by the above arguments,
$H^{i}(Y_{\mathbb{C}},
\mathcal{L}_{\mathbb{C}}^{e})=0$
for every $i > 0$ and $e \geq 0$. 
This proves (1).

%H^{i}(X_{\mathbb{C}},
%K_{X_{\mathbb{C}}}^{-e})^{G}$ 
%for every $i > 0$ and $e \geq 0$. 
%This proves (1).

Since $X_{p}$ is Frobenius split and $G$ is linearly reductive over $\overline{\mathbb{F}}_p$,  using 
Reynolds operator, we see that $Y_{p}$ is also Frobenius split. For a proof,
see \cite[Theorem 3.7]{Kan1}. 

It is well known that there is a positive integer $r$ such that 
$H^{i}(Y_{p}, \mathcal{L}_p^{-p^{r}})=0$ for $i\neq d$. 

Now the proof of (2) is similar to the proof of (1) using the Frobenius splitting property of 
$Y_p$.

%will use the Frobenius splitting property of $Y$ to prove 2.
%Tensoring the map $\mathcal{O}_{Y_{p}} \longrightarrow F_{*}\mathcal{O}_{Y_{p}}$
%by $\mathcal{L}^{-1}$ noting that $\mathcal{L}^{-1}\otimes
%F_{*}\mathcal{O}_{Y_{p}} \longrightarrow
%F_{*}F^{*}\mathcal{L}^{-1}=F_{*}\mathcal{L}^{-p}$ we see that the map
%$H^{i}(Y_{p}, \mathcal{L}^{-1}\longrightarrow H^{i}(Y_{p}, \mathcal{L}^{-p}$
%is injective. 

%Iterating this process, we conclude that the map 
%$H^{i}(Y_{p}, \mathcal{L}^{-1})\longrightarrow H^{i}(Y_{p}, \mathcal{L}^{-p^{r}})$ 
%is injective. 

%Thus, we have $H^{i}(Y_{p}, \mathcal{L}^{-1})=0$ for $i\neq d$.
%Using arguments as in proof of 1, we conclude that 
%$H^{i}(Y, \mathcal{L}^{-1})=0$ for $i\neq d$.

This completes the proof of theorem.
\end{proof}

%Let $W$ be the regular representation of $G$. Let $V$ be a $m$-fold direct sum
%of $W$. Then $V$ is also a representation of $G$. Hence $G$ operates on the 
%projective space $X=\mathbb{P}(V)$. Clearly, every line bundle on $X$ is 
%$G$-linearised (ref ??)

%By (cf Peter Slodowy, Kraft ??), $Pic({G \backslash X})$ is a subgroup of 
%$Pic(X)$ and so it is generated by an effective line bundle, say 
%$\mathcal{L}_{0}$.

%By the descent lemma of Kempf (cf ??), the pull back 
%$\pi^{*}(\mathcal{L}_{0})$ is of the form $\mathcal{O}(t)$ for some
%positive integer $t$  dividing $n$. 
%In particular, $t$ divides dim($V$).
  
%Let $G$ be a finite solvable group of order $n$. Let $W$ be the regular 
%representation of $G$. Let $V$ be a $m$-fold direct sum of $W$. Then $V$ 
%is also a representation of $G$. Hence $G$ operates on the projective 
%space $X=\mathbb{P}(V)$. Clearly, every line bundle on $X$ is 
%$G$-linearised (ref ??)

%By (cf Peter Slodowy, Kraft ??), $Pic({G \backslash X})$ is a subgroup of 
%$Pic(X)$ and so it is generated by an effective line bundle, say 
%$\mathcal{L}_{0}$.

%By the descent lemma of Kempf (cf ??), the pull back 
%$\pi^{*}(\mathcal{L}_{0})$ is of the form $\mathcal{O}(t)$ for some positive 
%integer $t$  dividing $n$. In particular, $t$ divides dim($V$).
 
\section{$N_p$ property}

Let $X$ be a flat scheme over
$\mathbb{Z}$.
We assume that the hypotheses stated at the begininning of Section 3
hold. For simplicity of notation in this section we use letters $X, Y$
and $\li$ to denote $X_{\mathbb{C}}, Y_{\mathbb{C}}$ and
${\li}_{\mathbb{C}}$ respectively.

In this section we prove a result on $N_p$ property for $\li$ using
Theorem \ref{np}. By
Remark \ref{rmk}, we need the assumption that $\li$ is ample and base
point free. 
By Lemma \ref{ample}, $\li$
is ample. 

We assume further that $\li$ is base point free. 
Let $d=dim(Y)$. 

In \cite[Theorem 1.3]{GP1}, the authors prove a strong general result on $N_p$ property for
an ample and base point line free bundle $\li$  on a projective variety
$X$
%\footnote{We thank Milena Hering for pointing this out to us.}. 
The bounds
obtained are expressed in terms of the dimension of $X$ and regularity of
$\li$. Regularity of $\li$ is a measure of vanishing of higher
cohomology of powers of $\li$. This result is generalized to
multigraded regularity in  
\cite[Theorem 1.1]{HSS}.

Our next theorem follows from the above mentioned results. However we
provide a proof here for the sake a self-contained exposition. 

%We use the same notation as in the previous section: $X$ represents
%the projective space and $Y$ the quotient of $X$ by a finite group
%$G$. Let $\li$ be the generator of Pic$(Y)$. Let $d$ be the dimension
%of $Y$. 

%For any globally generated line bundle $L$, we have the following 
%exact sequence
%\begin{eqnarray}\label{cansurj}
%0 \rightarrow M_{L} \rightarrow H^0(Y,L) \otimes {\mathcal O}_{Y} \rightarrow L \rightarrow 0.
%
%\end{eqnarray}
%{\bf Note: Is $p$ confusing with prime from Section 3?? Some clarification needed?}
\begin{theorem}\label{ls}
Let $m, a \ge 1$ be a positive
integers. Then we have $H^i(Y,M^{\otimes m}_{\li^{\otimes a}} \otimes {\li}^{b-i}) =
0$ for $i \ge 1$ and $b > m+d$. 
\end{theorem}

\begin{proof}
We proceed by induction on $m$. 
%and use Lemmas \ref{horace} and
%\ref{key}. 

Let $m =1$. Let $a \ge 1$ and $b > d+1$. We first show that 
$H^1(Y,M_{\li^{\otimes a}} \otimes {\li}^{b-1}) =
0$.

Consider the sequence \ref{cansurj} with 
$F = \li^{\otimes a}$:
\begin{eqnarray}\label{cansurj1}
0 \rightarrow M_{\li^{\otimes a}} \rightarrow H^0(\li^{\otimes a}) \otimes {\mathcal O}_{Y} \rightarrow \li^{\otimes a} \rightarrow 0.
\end{eqnarray}

Tensoring with $\li^{\otimes b-1}$ and taking cohomologies, we get 
$$H^0(\li^{\otimes b-1}) \otimes H^0(\li^{\otimes a}) \xrightarrow{\alpha}
H^0(\li^{\otimes b+a-1}) \rightarrow H^1(M_{\li^{\otimes a}} \otimes
  \li^{\otimes b-1}) \rightarrow H^0(\li^{\otimes a}) \otimes
  H^1(\li^{\otimes b-1}) .$$

Since $H^1(\li^{\otimes b-1}) = 0$ by Theorem \ref{van}(1), if 
the map $\alpha$ is surjective then $H^1(M_{\li^{\otimes a}} \otimes
  \li^{\otimes b-1}) = 0$.

To prove surjectivity of $\alpha$, we will use Lemma \ref{horace} and first prove that the following map is surjective:  
$$\alpha_1: H^0(\li^{\otimes b-1}) \otimes H^0(\li^{}) \rightarrow
H^0(\li^{\otimes b}).$$
For this we use Lemma \ref{key}. The needed vanishings are
$H^i(\li^{b-1-i})=0,$ for $i = 1,\ldots,d$. These follow from 
Theorem \ref{van}(1) because $b > d+1$. 

Similarly, we obtain the surjectivity of the maps:
$$\alpha_2: H^0(\li^{\otimes b}) \otimes H^0(\li^{}) \rightarrow
H^0(\li^{\otimes b+1}),$$
$$\alpha_3: H^0(\li^{\otimes b+1}) \otimes H^0(\li^{}) \rightarrow
H^0(\li^{\otimes b+2}),$$ and so on. 
By Lemma \ref{horace}, $\alpha$ is surjective. 

For $i > 1$, we twist \eqref{cansurj1} by $\li^{\otimes b-i}$
and consider the long exact sequence in cohomology
:
$$H^{i-1}(\li^{\otimes a+b-i}) \rightarrow  
H^i(M^{}_{\li^{\otimes a}} \otimes {\li}^{\otimes b-i})
\rightarrow 
H^0(\li^{\otimes a}) \otimes
  H^i(\li^{\otimes b-i})$$

We get the desired vanishing by Theorem \ref{van}(1).

Now let $m > 1$ and suppose that the theorem holds for $m-1$. 
Let $a \ge 1$ and $b > m+d$ be given. First let $i=1$. 

Tensor the sequence \ref{cansurj1} with 
$M^{\otimes (m-1)}_{\li^{\otimes a}} \otimes {\li}^{b-1}$ and take
cohomology:

$H^0(M^{\otimes (m-1)}_{\li^{\otimes a}} \otimes {\li}^{\otimes b-1}) \otimes
H^0({\li}^{\otimes a}) \xrightarrow{\alpha} H^0(M^{\otimes (m-1)}_{\li^{\otimes a}}
\otimes {\li}^{\otimes a+b-1})
\rightarrow H^1(M^{\otimes m}_{\li^{\otimes a}} \otimes
{\li}^{\otimes b-1})
\rightarrow H^0(\li^{\otimes a}) \otimes H^1(M^{\otimes (m-1)}_{\li^{\otimes a}} \otimes {\li}^{\otimes b-1}).$

The last term is zero by induction hypothesis. Note that the
hypothesis required for $b$ holds. Hence it suffices to show that 
$\alpha$ is surjective. 

In order to show that $\alpha$ is surjective we will use Lemma 
\ref{horace} and first consider the
following
$$\alpha_1: H^0(M^{\otimes (m-1)}_{\li^{\otimes a}} \otimes {\li}^{b-1}) \otimes
H^0({\li}^{}) \rightarrow H^0(M^{\otimes (m-1)}_{\li^{\otimes a}}
\otimes {\li}^{\otimes b}).$$

By Lemma \ref{key}, this map surjects if 
$H^j(M^{\otimes (m-1)}_{\li^{\otimes a}} \otimes {\li}^{b-1-j}) = 0$ for $j
= 1,\ldots,d$. Since 
$b-1 > m-1+d$, the required vanishing
is clear from induction hypothesis applied to $m-1$. 

Now consider the map: 
$$\alpha_2: H^0(M^{\otimes (m-1)}_{\li^{\otimes a}} \otimes {\li}^{b}) \otimes
H^0({\li}^{}) \rightarrow H^0(M^{\otimes (m-1)}_{\li^{\otimes a}}
\otimes {\li}^{\otimes b+1}).$$
Using Lemma \ref{key}, as we did for $\alpha_1$, we conclude that
$\alpha_2$ is surjective too. Iterating this we obtain surjectivity of $\alpha_i$ for all $i$ and hence $\alpha$ is also surjective.

Finally for $i >1$, we have:
$$H^{i-1}(M^{\otimes (m-1)}_{\li^{\otimes a}}
\otimes {\li}^{\otimes a+b-i}) \rightarrow 
H^i(M^{\otimes m}_{\li^{\otimes a}} \otimes
{\li}^{\otimes b-i})
\rightarrow 
H^0(\li^{\otimes a}) \otimes H^i(M^{\otimes (m-1)}_{\li^{\otimes a}}
\otimes {\li}^{\otimes b-i})
$$

The middle term is zero because the other two are zero by
induction. 

%$H^i(M^{\otimes m}_{\li^{\otimes a}} \otimes
%{\li}^{\otimes b})=0$ by sandwitching it between 
%$H^{i-1}(M^{\otimes (m-1)}_{\li^{\otimes a}}
%\otimes {\li}^{\otimes a+b})$ and 
%$H^0(\li^{\otimes a}) \otimes H^i(M^{\otimes (m-1)}_{\li^{\otimes a}}
%\otimes {\li}^{\otimes b}).$ Both these are zero by induction.

%Next we proceed using Lemma \ref{horace} as in the $p=0$
%case. Induction hypothesis gives the required vanishing and we obtain
%that $\alpha$ is surjective. 

This completes the proof of the theorem. 
\end{proof}

\begin{corollary}\label{main1}
Let $X$ be a flat scheme over $\mathbb{Z}$ and let $G$ be a reductive
group scheme over $\mathbb{Z}$ acting on $X$. Suppose that all the
hypotheses stated at the beginning of Section 3 hold. Let $\li$ denote
the descent to $Y_{\complex}$ of the ample line bundle $\mathcal{N}_{\complex}$ on $X_{\complex}$. 

Then 
${\li}^{\otimes a}$ has $N_p$ property for $a > p+d$. 
\end{corollary}
\begin{proof}
By Theorem \ref{np}, ${\li}^{\otimes a}$ has $N_p$ property if 
$H^1(M^{\otimes m}_{\li^{\otimes a}} \otimes L^{\otimes an}) = 0$ for all $1 \le m \le p+1$ and $n \ge 1$. 

We apply the above theorem with $m = 1, \ldots, p+1$ and the required vanishing 
 follows immediately. 
\end{proof}

We get a stronger result when we assume that the top cohomology of
the structure sheaf vanishes. 

\begin{corollary}\label{main}Let $X$ be a flat scheme over $\mathbb{Z}$ and let $G$ be a reductive
group scheme over $\mathbb{Z}$ acting on $X$. Suppose that all the
hypotheses stated at the beginning of Section 3 hold. Let $\li$ denote
the descent to $Y_{\complex}$ of the ample line bundle $\mathcal{N}_{\complex}$ on $X_{\complex}$. 
Suppose further that $H^d(Y_{\complex},\mathcal{O}_{Y_{\complex}})=0$. 

Then ${\li}^{\otimes a}$ has $N_p$ property for $a \ge p+d$. 
\end{corollary}
\begin{proof}
This follows from the proof of Theorem \ref{np}. In the base case
($m=1)$, we required vanishing $H^i(\li^{b-i})=0,$ for $i =
1,\ldots,d$. For $i < d$, we apply Theorem \ref{van}(1). For $i = d$,
we use the hypothesis on the structure sheaf. 
\end{proof}

%\begin{remark}
%The results of this section were proved using only 
%Theorem \ref{van}(1). Thus these $N_p$ results hold for any pair
%$(Y,\li)$, where $Y$ is a complex projective variety and $\li$ is an
%ample and 
%base point free line bundle on $Y$, such that the statement of Theorem \ref{van}(1) is
%valid. 
%\end{remark}

\section{GIT quotients for the action of a finite group on a projective space}

Let $G$ be a finite group of order $n$. Let $\rho: G \rightarrow
GL(V)$ be a 
representation of $G$ over $\mathbb{C}$. 
$G$ operates on the projective space $X=\mathbb{P}(V)$ and every line
bundle 
on $X$ is $G$-linearised. Let $d$ be the dimension of $X$.

%$K_{X}^{-1}=\mathcal{O}(n)$. 

Note that for every point $x \in X$, 
the isotropy subgroup $G_{x}$ of $x$ in $G$
acts trivially on the fiber of $x$ in $\mathcal{O}_X(n)$. 
Hence, by \cite[Prop 4.2, Page 83]{Kraft}, 
$\mathcal{O}_X(n)$  
descends to the GIT 
quotient  $Y = G\backslash\backslash X_{G}^{ss}(\mathcal{O}_X(n))$. Let $\li$
denote the descent of $\mathcal{O}_X(n)$ to $Y$. 

Let $x \in X$.
Since $G$ is finite, there is a $s \in H^0(X,\mathcal{O}_X(1))$ such
that $s(gx) \ne 0$ for all $g \in G$. 
Let $\sigma = \prod_{g \in G} g.s$. Then $\sigma \in
H^0(X,\mathcal{O}_X(n))^G$ and $\sigma(x) \ne 0$. Hence
$\li$ is base point free. Also note that $\li$ is ample by 
Lemma \ref{ample}.

Note that the statement of Theorem \ref{van}(1) holds in this
case. The higher cohomologies of nonnegative powers of
$\mathcal{O}_X(n)$ are clearly zero and hence the higher cohomologies of
nonnegative powers of $\li$ are zero too (cf.  \cite[Theorem 3.2.a]{Tel}).

%Note that there is a free $\mathbb{Z}$-submodule $M$ of $V$ of rank
%$n$ which is $G$-stable such that $M \otimes_{\mathbb{Z}} \mathbb{C} =
%V$. Note that $X_p$ is Frobenius split 
%(cf. \cite[Theorem 2, Page 38]{MR} or \cite[Theorem 2.2.5, Page 69]{BK}).

%Hence all the hypotheses in Section 3 hold in this setting and 
%we can apply results of Section 4. 

Hence, by the Corollary \ref{main}, we have the following.
\begin{theorem}
$\li^{\otimes  a}$ has $N_{p}$ property for any $a\geq p+d$. 
\end{theorem}
For $p=0$, we deduce the following corollary. This result is new 
compared to \cite{KPP} since it works for every group. This result is
also new compared 
to \cite{Kang} since the bound on degree is small.
 
\begin{corollary}
$\li^{\otimes d}$ has $N_{0}$ property.
\end{corollary}

Note that $\mathcal{O}_X(n)$ is the inverse of the canonical line bundle
$K_X$ of $X$. Since $K_X$ descends to $Y$, it is a
natural question to ask whether this descent is the canonical line bundle of
$Y$, if $Y$ is Gorenstein. We obtain the following result.

%In this paper  we are looking at the descent of the inverse of the canonical line bundle. It is natural to ask in the specific situation  of this section, what the descent of the canonical line bundle will be. We obtain the following result in a special case.

\begin{lemma}\label{det1}
Assume that $\rho (G)\subset SL(V)$.
Then 
the descent of the canonical line bundle $K_{X}$ of $X$ to 
$Y=G \backslash\backslash
X_{G}^{ss}(K_{X}^{-1})$  is the canonical line bundle of $Y$. 
\end{lemma}
\begin{proof}
Let $\li$ denote the descent of $K_X$ to $Y$. 
By \cite[Theorem 1]{Wat}, $\mathbb{C}[V]^{G}$ is Gorenstein (cf
\cite[Corollary 8.2]{Stan} also). 
Hence $Y$ is Gorenstein. Hence, the dualising sheaf of $Y$ is the 
canonical line bundle of $Y$. We denote it by $K_{Y}$. By \cite
[Theorem 3.2]{Tel},
there is a positive integer $h$ such that $H^{d}(Y, {\li}^{\otimes
  h})=H^{d}(X, K_{X}^{\otimes h})^{G}$.  
By applying Serre duality for the variety  $X$, we have 
$H^{d}(X, K_{X}^{\otimes h})^{*}=H^{0}(X, K_{X}^{\otimes (1-h)})$. By GIT, we 
have $H^{0}(Y, {\li}^{\otimes (1-h)})=H^{0}(X, K_{X}^{\otimes (1-h)})^{G}$.
On the other hand, applying Serre duality for the variety  $Y$, we have 
$H^{d}(Y, {\li}^{h})=H^{0}(Y, {\li}^{\otimes (-h)}\otimes K_{Y})^{*}$.

Summarising, we have $H^{0}(Y, {\li}^{\otimes
  (1-h)})=H^{0}(Y,{\li}^{\otimes (-h)}\otimes K_{Y})$. By 
\cite[Corollary 4.2, Page 83]{Kraft}, $Pic(Y)$  is a subgroup of $Pic(X) =
\mathbb{Z}$ and hence it is cyclic, 
say generated by $\li_0$. Since $\li$ and $K_Y$ are powers of $\li_0$,
the above equality of 
global sections shows that $\li=K_{Y}$. Hence the descent of $K_X$ to $Y$ is $K_Y$. 
\end{proof}

\section{GIT quotients for the action of a maximal
  torus 
on the 
flag variety}

For the preliminaries on semisimple algebraic groups, semisimple Lie
algebras and root systems, we refer to 
\cite{Hump1, Hump2}. For the preliminaries on Chevalley groups we
refer to \cite{Stein}. 
 
Let $G$ be a semisimple Chevalley group over $\mathbb{C}$ of rank $n$. 
Let $T$ be a maximal torus of 
$G$, $B$ a Borel subgroup of $G$ containing $T$, which are defined over $\mathbb{Z}$. 
Let $N_{G}(T)$ denote the normaliser of $T$ in $G$. Let $W=N_{G}(T)/T$
denote the Weyl group of $G$ with respect to $T$.

We note that $G, T, B, W$ are all defined over $\mathbb{Z}$. Hence the
flag variety $G/B$ of all Borel subgroups of $G$ and Schubert
varieties are also defined over
$\mathbb{Z}$ \cite[Page 21]{Stein}. Note that the base change of any
Schubert variety to $\overline{\mathbb{F}}_p$ is 
Frobenius split 
(cf \cite[Theorem 2, Page 38]{MR} or \cite[Theorem 2.2.5, Page 69]{BK}).

We denote by $\mathfrak{g}$ the Lie algebra of $G$. 
We denote by $\mathfrak{h}\subseteq \mathfrak{g}$ the Lie algebra of $T$.
Let $R$ denote the roots of $G$ with respect to $T$.
Let $R^{+}\subset R$ be the set of positive roots with respect to $B$.
Let $S=\{\alpha_{1}, \alpha_{2}, \cdots ,\alpha_n\}\subset R^{+}$ denote 
the set of simple roots with respect to $B$.
Let $\langle . , . \rangle$ denote the restriction of the Killing form to 
$\mathfrak{h}$.
Let $\check{\alpha}_i$ denote the coroot corresponding to $\alpha_{i}$.
Let $\varpi_1, \varpi_2,\cdots, \varpi_n$ denote the fundamental weights 
corresponding to $S$.

Let $s_{i}$ denote the simple reflection in $W$ corresponding to the simple root $\alpha_{i}$.
For any subset $J$ of $\{1,2,\cdots,n\}$, we denote by $W_J$ the subgroup of $W$ generated by $s_j$, $j \in J$. 
We denote the complement of $J$ in $\{1,2,\cdots,n\}$ by $J^{c}$. For each 
$w\in W$, we choose an element $n_{w}$ in $N_{G}(T)$ such that $n_{w}T=w$.
We denote the parabolic subgroup of $G$ containing $B$ and 
$\{n_{w}: w\in W_{J^c}\}$ by $P_{J}$. In particular, we denote the maximal 
parabolic subgroup of $G$ generated by $B$ and $\{n_{s_{j}} ; j \neq i \}$ by 
$P_i$.

%Let $w_0$ denote the longest element of $W$ corresponding to $B$.
%Let $B^-=w_0Bw_0^{-1}$ denote the Borel subgroup of $G$ opposite to $B$ with 
%respect to $T$.

Let $X(B)$ denote the group of characters of $B$
and let $\chi \in X(B)$. Then, we have an action of $B$ on $\mathbb C$, namely $b.k=\chi(b^{-1})k, \,\, b \in B, 
\,\, k \in \mathbb C$. Consider the equivalence relation $\sim$ on 
$G \times \mathbb C$ defined by $(gb, b.k) \sim (g, k), g \in G, b \in
B, k \in \mathbb C$. The set of all equivalence classes is the 
total space of a line bundle over $G/B$. We denote this $G$- linearised line bundle associated to $\chi$ by $\mathcal L_{\chi}$.

%We recall the notion of semi-stable and stable points introduced by Mumford. 

%A point $x$ in $X(w)$ is said to be semi-stable with respect to the
%$T$-linearised line bundle $\mathcal{L}_{\chi}$ if there is a
 %positive integer $m$ and a $T$-invariant section $s$ of $H^{0}(X(w),
 %\mathcal{L}^{\otimes m}_{\chi})$ such that $s(x)\neq 0$. 
%For more details, see [Mum]. We denote the set of all semi-stable 
%points in $X(w)$ with respect to the $T$-linearised line bundle 
%$\mathcal L_\chi$ by $X(w)^{ss}_T(\mathcal L_\chi)$.

%A point $x$ in $X(w)$ is said to be stable with respect to the
%$T$-linearised 
%line bundle $\mathcal{L}_{\chi}$ if $x$ is in $X(w)^{ss}_T(\mathcal
%L_\chi)$, the $T$- orbit of $x$ in $X(w)^{ss}_T(\mathcal L_\chi)$ is 
%closed and the isotropy subgroup of $x$ in $T$ is finite.

Let $G=SL(n+1, \mathbb{C})$.
Let $J$ be a subset of $\{1, 2, \cdots n\}$ and let $P_J$ be the
parabolic subgroup of $G$ corresponding to $J$. 
Since $G$ is simply connected, every line bundle on
$G/P_J$ is $G$-linearised (cf. \cite[3.3, Page 82]{Kraft}).

Let 
$W^{J^c}$ be the minimal representatives of elements in $W$ with
respect to the subgroup $W_{J^c}$. 
For $w \in W^{J^c}$, let $X(w) = \overline{BwP_J/P_J} \subset G/P_J$
be the Schubert variety corresponding to $w$. Note that $X(w)$ is
$T$-stable. Hence restriction of any line bundle on $G/P_J$ to $X(w)$
is $T$-linearised. 

Let $\chi$ be a dominant character of $T$  which is in the root
lattice such that $\langle \chi,\alpha_j\rangle > 0$ for every 
$j \in J$.  Let $w \in W^{J^c}$ be such that
$X(w)_{T}^{ss}(\li_{\chi})$ is nonempty.
By \cite[Theorem 3.10.a, Page 758]{Kum}, 
$\li_{\chi}$ descends to the GIT quotient  
$T\backslash\backslash X(w)_{T}^{ss}(\li_{\chi})$. Let $\mathcal{N}_{\chi}$ denote the 
descent. 

Since $\chi$ is in the root lattice, by \cite[Theorem 2.3]{Howard},
for every $x \in X(w)_{T}^{ss}(\li_{\chi})$, 
there is a $T$-invariant section $s$ of $\mathcal{L}_{\chi}$ such that $s(x) \ne 0$.
Hence $\mathcal{N}_{\chi}$ is base point free. Also
$\mathcal{N}_{\chi}$ is 
ample by Lemma \ref{ample}.

%Let $d = dim(X(w))-dim(T)$. We have
\begin{theorem}
Let $Y  = T\backslash\backslash X(w)_{T}^{ss}(\li_{\chi})$ be the GIT
quotient of $X(w)$ for the $T$-linearised line bundle $\li_{\chi}$ on $X(w)$.
Let $d$ be the dimension of $Y$.
Let $\mathcal{N}_{\chi}$ be the descent of $\li_{\chi}$
to $Y$. Then 
$\mathcal{N}_{\chi}^{\otimes a}$ has $N_p$ property for $a \geq p+d$. 
\end{theorem}
\begin{proof}
This follows from the above discussion and Corollary \ref{main}.
\end{proof}

Now let $X = G/P_J$. 
We apply this theorem to the inverse of the canonical line bundle
$K_X$ of
$X$.

Let $R_{J}^{+}$ denote the set of all positive roots $\beta$
satisfying $\beta\geq \alpha_{j}$ for some $j \in J$.
Let $\chi_{J}$ be the sum of all elements in $R_{J}^{+}$.
Then, by equality (6) in \cite[Page 229]{Jan},
we have, $K_{X}^{-1}=\mathcal{L}_{\chi_{J}}$. Note that 
$\langle \chi_J,\alpha_j\rangle > 0$ for every $j \in J$. Hence, by 
\cite[Remark 1, Page 232]{Jan}, $K_{X}^{-1}$ is ample. 

By using similar arguments as above,  $K_X^{-1}$ descends
to the GIT quotient $T\backslash\backslash
X_{T}^{ss}(K_{X}^{-1})$. Let $\li$ denote the descent. Again using
similar arguments as above, we see that $\li$ is ample and base
point free.

%Since $\chi_J$ is in the root lattice, by \cite[Theorem 3.10.a, page 758]{Kum}, $K_{X}^{-1}$ descends to the GIT quotient  
%$T\backslash\backslash X_{T}^{ss}(K_{X}^{-1})$. Let $\mathcal{L}$ denote the 
%descent. 

%Since $\chi_J$ is in the root lattice, by \cite[Theorem 2.3]{Howard},
%for every $x \in X_{T}^{ss}(K_{X}^{-1})$, 
%there is a $T$-invariant section $s$ of $K_{X}^{-1} =
%\mathcal{L}_{\chi_{J}}$ such that $s(x) \ne 0$.
%Hence $\li$ is base point free. Also $\li$ is ample by Lemma \ref{ample}.

Let $d = dim(X)-dim(T)$. We have 

\begin{corollary}
$\mathcal{L}^{\otimes a}$ has $N_p$ property for $a \geq p+d$. 
\end{corollary}
\begin{remark}
For simple algebraic groups $G$ of types different from $A_n$ canonical line bundle of the flag variety $G/B$ does not, in general, descend to the GIT quotient. For instance,  if $G$ is of type $B_3$, the coefficient of simple root $\alpha_1$ in the expression of $2\rho$ is 5. By
\cite[Theorem 3.10.b, Page 758]{Kum}, we see that the canonical line bundle of $G/B$ does not descend to the GIT quotient 
$T\backslash\backslash (G/B)_{T}^{ss}(\li(2\rho))$.
\end{remark}
{\bf Acknowledgements} We thank T.R. Ramadas for helpful
discussions. We thank Milena Hering for bringing to our notice the
results on regularity in \cite{GP1} and \cite{HSS}. These results enabled us to get better
bounds in our results. 

\end{document}